\documentclass{amsart}
\usepackage{color,enumitem,graphicx}
\usepackage{amsthm}
\usepackage{amsmath,amssymb}
\usepackage{mathrsfs}
\usepackage{bigints}
\usepackage[breaklinks]{hyperref}
\newtheorem{thm}{Theorem}[section]
\newtheorem{cor}[thm]{Corollary}
\newtheorem{lem}[thm]{Lemma}

\theoremstyle{definition}
\newtheorem{defn}[thm]{Definition}
\theoremstyle{remark}
\newtheorem{rem}[thm]{Remark}
\theoremstyle{example}
\newtheorem{ex}[thm]{Example}
\numberwithin{equation}{section}

\DeclareMathOperator{\ess}{ess}

\newcommand{\abs}[1]{\left\vert#1\right\vert}

\newcommand{\norm}[1]{\left\Vert#1\right\Vert}

\begin{document}

	\title[Three non-zero solutions of a Neumann eigenvalue problems involving the fractional p-Laplacian  ]
	{Three non-zero solutions of a Neumann eigenvalue problems involving the fractional p-Laplacian  }
	\author[S.\ Gandal, J.\,Tyagi]
	{ Somnath Gandal }
	\address{Somnath Gandal \hfill\break
		Department of Mathematics, Central University of Rajasthan, \newline
		NH-8, Bandarsindari, District-Ajmer, Rajasthan, Pin-305817,
		India }
	\email{somnathg@curaj.ac.in, gandalsomnath@gmail.com}
	\date{\today}
	\thanks{Submitted \today.  Published-----.}

	\begin{abstract}
		In the present paper, we establish a multiplicity result for a following class of nonlocal Neumann eigenvalue problems involving the fractional p-Laplacian.
		\begin{align*}  \begin{cases}
				(-\Delta)^{s}_{p}u + a(x) \abs{u}^{p-2}u =\lambda h(x,u) & \text {in } \Omega, \\
				\mathcal{N}_{s,p}u=0  & \text {in } \mathbb{R}^N \setminus \overline{\Omega}, 
			\end{cases}
		\end{align*}
		Precisely, we demonstrate the existence of an open interval for positive eigenvalues $\lambda$, for which the problem has at least three non-zero solutions in $W^{s,p}_{\Omega}.$
	\end{abstract}
	
	\maketitle
	
	\tableofcontents
	
	\section{Introduction}
	We investigate the existance and multiplicity of solutions to the following problem:
	\begin{align} \label{P1} \tag{P1}	\begin{cases}
			(-\Delta)^{s}_{p}u + a(x) \abs{u}^{p-2}u =\lambda h(x,u) & \text {in } \Omega, \\
			\mathcal{N}_{s,p}u=0  & \text {in } \mathbb{R}^N \setminus \overline{\Omega}, 
		\end{cases}
	\end{align}
	where $\Omega \subset \mathbb{R}^N$ is a non-empty bounded open subset with boundary of class $C^1$, $0<s<1,$ $ 1 \leq  N,$ and $h: \Omega \times \mathbb{R} \rightarrow \mathbb{R}$ is a Carath\'{e}odory function with the following subcritical growth condition
	\begin{align}\label{caraclass}
		|h(x, t)| \leq a_1+a_2|t|^{q-1}, \quad \forall(x, t) \in \Omega \times \mathbb{R}
	\end{align}
	for some non-negative constants $a_1, a_2$ and $q \in \left(1, p_{s}^{*} \right),$ where 
	\begin{align*}
		p_{s}^{*}:= \begin{cases}
			\frac{Np}{N-ps} & \text{ if } N> ps,\\
			+\infty & \text{ if } N\leq ps.
		\end{cases}
	\end{align*}
	Moreover, $a \in L^{\infty}(\Omega)$, with $ \displaystyle{\operatorname{essinf}_{x \in \Omega}} a(x)>0, \lambda$ is a positive real parameter. We denote by $\mathfrak{F}$ the class of all Carath\'{e}odory functions that satisfy the condition \ref{caraclass}. 
	
	In recent years, the study of nonlocal problems involving the fractional p-Laplacian $	(-\Delta)_{p}^s $ has garnered significant attention from analysts due to its mathematical importance and wide range of applications. For $s\in \left(0,1 \right)$ and $p \in (1,\, + \infty),$ the fractional p-Laplace operator $	(-\Delta)_{p}^s $ is defined as follows:
\begin{align}
	(-\Delta)_{p}^s u(x):=C_{N,s,p} P.V. \int_{\mathbb{R}^N} \frac{|u(x)-u(y)|^{p-2}(u(x)-u(y))}{|x-y|^{N+sp} } dy, \end{align}
where P.V. stands for principal value and $C_{N,s,p}$ is the normalizing constant. In the context of nonlocality, at least three distinct types of fractional normal derivatives have been explored. The first is of a geometric nature (see \cite{Ros1,Ros2,Fall}), the second emerges from spectral notion of the fractional Laplace operator (see \cite{Mont,Sting}), and the third is $\mathcal{N}_{s, p},$ which aligns well with energy methods and variational techniques. It is defined as follows,
\begin{align}
	\mathcal{N}_{s, p} u(x):=C_{N,s,p} \int_{\Omega} \frac{|u(x)-u(y)|^{p-2}(u(x)-u(y))}{|x-y|^{N+sp}} d y, \quad x \in \mathbb{R}^N \backslash \overline{\Omega}.
\end{align}
The Neumann condition $	\mathcal{N}_{s, p}u$ is recently introduced for $p=2$ in \cite{Dip1} and generalized in \cite{Barr2, Mug1}. When \( p = 2 \), we denote \( (-\Delta)_p^s \) and \( \mathcal{N}_{s,p} \) simply as \( (-\Delta)^s \) and \( \mathcal{N}_s \), respectively. For the sake of simplicity, now on, we set $C_{N,s,p} =1.$  This Neumann condition leads to a variational framework for the corresponding problems involving fractional p-Laplace operators, see \cite{Barr2,Dip1,Gan1,Gan2,Mug1}. It is worth noting that \( \mathcal{N}_{s, p} \) is defined not at the boundary points, but for points residing outside the domain $\Omega .$ The operator $\mathcal{N}_{s, p}$ possesses a natural probabilistic interpretation, rendering it particularly appropriate for modeling and analysis in mathematical biology. 
The operator $ (-\Delta)^{s}_{p} $ is an infinitesimal generator of a L\'{e}vy process. A L\'{e}vy process is a stochastic process characterized by stationary and independent increments, but its paths may not necessarily be continuous. This process models the random movement of a particle over time, for $ t > 0.$ Let $ u(x, t) $ represent the probability density function of the particle's position. The function $ u $ satisfies the equation
$$u_t + (-\Delta)^{s}_{p} u = 0 \quad \text{in} \quad \mathbb{R}^N,$$
indicating that if the particle is located at $x \in \mathbb{R}^N ,$ it can``jump" to any other point $ y \in \mathbb{R}^N ,$ with the probability density of jumping to $ y $ being proportional to $ |x - y|^{-(N - sp)} .$ In simple terms, this operator describes diffusion or spread, but unlike the classical Laplacian, it accounts for the influence of distant points in space. This means that the ``diffusion" or ``spread" of $u$ is not limited to immediate neighbors, but extends over a broader region. The parameter $s$ controls the degree of nonlocality, with $s=1$ corresponding to the usual p-Laplacian and  $s<1$ describing a more spread-out or long-range interaction.

Now, let $ u(x, t) $ represent the probability density function of the particle's position as it moves randomly within a domain $ \Omega.$ When the particle exits the domain $ \Omega,$ it immediately returns to $ \Omega .$ The way in which the particle comes back into $ \Omega $ is as follows: If the particle exits $ \Omega $ and reaches a point  $ x \in \mathbb{R}^N \setminus \overline{\Omega},$ it may return to any point $y \in \Omega ,$ with the probability density of jumping from $x $ to $ y $ being proportional to $ |x - y|^{-(N - sp)} .$ The problem is then described by the following system of equations:
\begin{align}\label{p3}
\begin{cases}
	u_t + (-\Delta)^{s}_{p} u = 0 & \text{in } \Omega, \\
	u(x, 0) = u_0(x) & \text{in } \Omega, \\
	\mathcal{N}_{s,p} u = 0 & \text{in } \mathbb{R}^N \setminus \overline{\Omega},
\end{cases}
\end{align}
where $ \mathcal{N}_{s,p} u $ represents the appropriate boundary condition for the particle returning to $ \Omega .$ We refer to \cite{Dip1} for detailed description of these models in the case $p=2.$ When the fractional p-Laplacian $(- \Delta)^{s}_{p}$ in equation \ref{p3} is replaced by a superposition of the classical p-Laplacian and the fractional p-Laplacian, with either local or nonlocal Neumann boundary conditions, a new evolution equation emerges. In \cite{Dip3}, Dipierro and Valdinoci provide a compelling motivation for studying such models. This model captures the diffusion of a biological population within an ecological environment, influenced by both classical diffusion and long-range dispersal mechanisms. We also refer to \cite{Dip4} which deals with the sudy of problems involving mixed local and nonlocal operaotrs.\\
In most of the situations found in nature, the rate of change $u_t $ describes how a quantity evolves over time, influenced by both diffusion and interaction effects.  In these cases, $u $ satisfies the equation 
\begin{align}
	\begin{cases}
		u_t + (-\Delta)^{s}_{p} u = g(x,u) & \text{in } \Omega, \\
		u(x, 0) = u_0(x) & \text{in } \Omega, \\
		\mathcal{N}_{s,p} u = 0 & \text{in } \mathbb{R}^N \setminus \overline{\Omega},
	\end{cases}
\end{align}
 where $ g(x,u) $ represents the nonlinear interaction or reaction term, which depends on both the position $ x $ and the quantity $u $ itself. For different choices of $ g,$ researchers have studied the existence and qualitative properties of the solutions to such problems in the local case, i.e. when $s=1.$ The question of steady-state solutions to these problems is an interesting topic from both an application and mathematical perspective. In a biological or physical context, a steady-state solution represents a system that has reached equilibrium. For example:\\
 \textbf{Population Dynamics:} In a population model, a steady-state could represent the situation where the population density at each location remains constant over time because the birth and death rates (modeled by 
$g(x,u)$) and the movement of individuals (modeled by 
 $ (-\Delta)^{s}_{p} u  $ ) have balanced out.\\
 \textbf{Chemical Reactions:} In chemical reaction models, a steady-state solution could represent a situation where the concentration of a chemical substance stabilizes, with the rate of its diffusion matching the rate of its production or consumption.\\
Motivated by the significance and mathematical importance of the nonlocal operator $(-\Delta)_{p}^{s}$ and nonlocal normal condition $\mathcal{N}_{s,p}$, we investigate the existance and multiplicity of solutions to the problem \ref{P1}. In the local case where $s = 1,$ the existence of multiple solutions for the Neumann problem has been extensively studied in the literature. Notably, in lower dimensions $(p > N ),$ this has been explored in \cite{Averna, Bonano3, Bonano4, Bonano5}, while in higher dimensions $(N > p ),$ the topic has been investigated in \cite{Dagui, Faraci, Ricceri1, Ricceri2}.  Since the nonlocal Neumann condition \( \mathcal{N}_{s,p} \) was recently introduced there is limited knowledge about the problems involving it.  Also the nonlocal nature of both the operator and the boundary conditions likely play a significant role in this limited exploration. To highlight works on the fractional p-Laplacian and nonlocal Neumann conditions, I would like to reference \cite{Mug1}, where Mugnai and Lippi developed a variational framework for the nonlocal p-Neumann derivative. They proved a maximum principle and established an unbounded sequence of eigenvalues. While this does not constitute a complete spectral notion for the p-fractional Laplacian with the p-Neumann derivative, some classical properties of the set of eigenvalues for the p-Laplacian remain valid in this context. Furthermore, in \cite{Mug}, Mugnai, Perera, and Lippi studied the regularity of solutions to the following problem:
\begin{align}
\begin{cases}
	(-\Delta)_{p}^{s} u = f(x,u) & \text{in } \Omega, \\
	\mathcal{N}_{s,p} u= 0 & \text{in } \mathbb{R}^N \setminus \overline{\Omega},
\end{cases}
\end{align}
subject to specific conditions on the growth of  $f.$ Using the notion of linking over cones, Mugnai and Lippi \cite{Mug} proved the existence of solutions to the fractional problem:
\begin{align}
\begin{cases}
	(-\Delta)_p^s u = \lambda |u|^{p-2} u + g(x,u) & \text{in } \Omega, \\
	\mathcal{N}_{s,p} u= 0 & \text{in } \mathbb{R}^N \setminus \overline{\Omega},
\end{cases}
\end{align}
where \(\lambda \geq 0\) and \(g\) satisfies certain conditions. Now, I would like to draw the readers' attention to the case \( p = 2 \), which corresponds to the fractional Laplacian \( (-\Delta)^s \). Consider a following problem:
\begin{align}\label{p=2case}
	\begin{cases}
		(-\Delta)^{s} u + \lambda u= h(u) & \text{in } \Omega, \\
		\mathcal{N}_{s}u = 0 & \text{in } \mathbb{R}^N \setminus \overline{\Omega},
	\end{cases}
\end{align}
Barrios, Del Pezzo, and García-Melián \cite{Barr2} explored the existence of non-negative solutions to \eqref{p=2case} in the subcritical case, i.e., when $ |h(u)| = O(|u|^q),$ with $ 1 \leq q < \frac{N+2s}{N-2s}.$ Further, Gandal and Tyagi \cite{Gan2} obtained the non-negative solutions of \eqref{p=2case} in the critical case, i.e., when $ q = \frac{N+2s}{N-2s}$ and, also proved their uniqueness in small domains. The asymptotic behavior of solutions to \eqref{p=2case} is an interesting aspect to explore. By asymptotic behavior, we refer to how the solutions behave as $\lambda$ becomes very large. In the local case, i.e., when $s = 1,$ the geometry of the domain plays a crucial role in analyzing the shape of the least energy solutions, see \cite{Adi,Ni}. Specifically, the curvature of the boundary helps in determining the location of the maximum point of solutions. The nonlocal nature of the fractional Laplacian and the boundary conditions add complexity to the study of this question to problem \eqref{p=2case}. The question of asymptotic behavior of solutions to \eqref{p=2case} has been partially addressed for the subcritical case by Gandal and Tyagi in \cite{Gan1}, while it remains completely open in the critical case.  Using different approach, Cinti and Colasuonno \cite{Cinti1,Cinti} established the existence of non-negative radial solutions to \eqref{p=2case}, when $\Omega$ is either a ball or annulus and $h$ is a superlinear nonlinearity. Readers can also explore the recent citations of the articles mentioned above for more in-depth and valuable information on nonlocal Neumann problems. \\
To the best of my knowledge, this is the first paper to address the multiplicity of nonlocal Neumann problems involving the fractional p-Laplacian. The current work in this paper extends the local multiplicity results established in \cite{Averna,Dagui} for problems involving the p-Laplace operator to the fractional p-Laplace operator.  Before presenting our main results, we first make the necessary functional setting as follows.\\
\noindent \textbf{Functional framework:}  Let $T(\Omega):= \mathbb{R}^{2N} \setminus \left( \mathbb{R}^N \setminus \Omega \right)^2$ be a cross shaped set on $\Omega.$ For a measurable function $u : \mathbb{R}^N \rightarrow \mathbb{R} $ set  
 $$\norm{u}:= \left( \int_{\Omega} a(x) \abs{u}^p dx + \frac{1}{2}\int_{T(\Omega)} \frac{|u(x)-u(y)|^{p}}{|x-y|^{N+sp}} dx d y \right)^{\frac{1}{p}}.$$ Define 
 $W^{s,p}_{\Omega}:= \left\{ u : \mathbb{R}^N \rightarrow \mathbb{R} \text{ measurable } \mid \norm{u}<\infty \right\}.$ Note that $W^{s,p}_{\Omega}$ is a reflexive Banach space, and when $p=2,$ it is a Hilbert space with natural inner product induced by the above norm. We now recall the Sobolev embedding results. Let $\Omega \subset \mathbb{R}^N$ be an open bounded set of class $C^{1}.$ Then we have the following embeddings:
 \begin{enumerate}
 	\item[(i)] If $N>sp,$ $ W^{s,p}_{\Omega} \hookrightarrow L^q(\Omega)$ is continuous for $1\leq q \leq p^{*}_{s}$ and compact for $1\leq q < p^{*}_{s}.$ i.e., there exists a positive constant $c_q$ such that
 	\begin{align}\label{sobolevconst}
 		\|u\|_{L^q(\Omega)} \leq c_q\|u\|,
 	\end{align}
 	for every $u \in W^{s,p}_{\Omega} $. \\
 	\item[(ii)] If $N=sp,$ $ W^{s,p}_{\Omega} \hookrightarrow L^q(\Omega)$ is compact embedding for all $1 \leq q< + \infty $ and hence the above inequality holds in this case as well. \\
	\item[(iii)] If $N<sp,$ $ W^{s,p}_{\Omega} \hookrightarrow C(\overline{\Omega})$  is compact.
 \end{enumerate}
  We refer to \cite{Din1,Mug1} for the proofs of the embedding results and more about the fractional Sobolev inequalities. The solutions to problem \eqref{P1} are understood to be weak solutions, which are defined as follows.
\begin{defn}
$u \in W^{s,p}_{\Omega}$ is said to be a weak solution to the problem \eqref{P1} if it satisfies the following equation:
$$
\frac{1}{2} \int_{T(\Omega)} \frac{\abs{u(x)-u(y)}^{p-2}(u(x)-u(y))(v(x)-v(y))}{|x-y|^{N+p s}} d x d y + \int_{\Omega} a(x) \abs{u}^{p-2}u v \, dx =\lambda \int_{\Omega} h(x,u) v \,d x,
$$
for every $v \in W^{s,p}_{\Omega}.$
\end{defn}

\noindent Let us define the functional $J_{\lambda} : W^{s,p}_{\Omega} \rightarrow \mathbb{R}$ as follows.
\begin{align}
	J_{\lambda}(u)= \frac{1}{p} \left(  \frac{1}{2}\int_{T(\Omega)} \frac{|u(x)-u(y)|^{p}}{|x-y|^{N+sp}} dx d y  +  \int_{\Omega} a(x)\abs{u}^p dx \right)- \lambda \int_{\Omega} \left( \int_{0}^{u(x)} h(x, t) dt \right) dx.
\end{align}
Since the embedding $ W^{s,p}_{\Omega} \hookrightarrow L^q(\Omega)$ is continuous, the functional $J_{\lambda}$ is well defined and of class $C^{1}$ on $W^{s,p}_{\Omega}.$ Note that any critical point of $J_{\lambda}$ is a weak solution to \eqref{P1}. 
Further, let
$$
H(x, \xi):=\int_0^{\xi} h(x, t) d t,
$$
for every $(x, \xi) \in \Omega \times \mathbb{R}.$ We investigate the existence of multiple solutions in two distinct cases naturally arising from the Sobolev embedding theorems.

\noindent \textbf{Case I:} First we consider the case $1 \leq N<sp,$ $p \geq 2.$
In this case, we know that the space $W^{s,p}_{\Omega}$ is compactly embedding in  $C(\overline{\Omega}),$ so that \\
\begin{align}\label{bestconst}
c:=\sup_{u \in W^{s,p}_{\Omega}\setminus \left\{0\right\}} \frac{\norm{u}_{L^{\infty}(\Omega)}}{\norm{u}} < \infty.
\end{align}
Clearly, $c^p \norm{a}_{L^{1}(\Omega)} \geq 1.$

\begin{thm} \label{mainresult1}
	Let $1\leq N<sp,$  $p \geq 2$ and $h: \Omega \times \mathbb{R} \rightarrow \mathbb{R}$ be a Carath\'{e}odory function such that for all $R>0,$ $\sup_{\abs{u} \leq R} \abs{h(\cdot,u)} \in L^{1}(\Omega).$  Assume that 
	\begin{enumerate}
	\item[(Ah1)]there exist a positive constant $t<p,$ and  a function $\mu \in L^{1}(\Omega)$ such that $$H(x,\xi) \leq \mu(x)\left(1 + \abs{\xi}^{t}\right) $$ for almost every $x \in \Omega$ and $\xi \in \mathbb{R};$\\
	\item[(Ah2)] there exist two positive constants $\gamma$ and  $\eta,$ with $\eta > \gamma,$ such that  $$\frac{\int_{\Omega} H(x, \eta) d x}{\eta^p} > \frac{1+c^p \norm{a}_{L^{1}(\Omega)}}{\Gamma},$$ where $\Gamma:=\int_{\Omega} \sup_{\abs{\xi}\leq\gamma } H(x,\xi) dx.$  Then, for each parameter $\lambda$ belonging to $$\left(\frac{\eta^p \norm{a}_{L^{1}(\Omega)}}{p\int_{\Omega} H(x,\eta)dx-p\Gamma},\, \frac{\gamma^p}{p c^p \Gamma} \right), $$ the problem \eqref{P1} possesses at least three weak solutions in $W^{s,p}_{\Omega}.$
	\end{enumerate}
\end{thm}

\noindent \textbf{Case II:} Next we consider the case $ N \geq sp \geq 1,$ $p>1.$
We set
$$
\kappa:=\left(\frac{p}{\|a\|_{L^1(\Omega)}}\right)^{\frac{1}{p}}
$$
and
$$
L_1:=\frac{c_1}{p^{\frac{p-1}{p}}}\|a\|_{L^1(\Omega)}, \quad L_2:=\frac{c_q^q}{q p^{\frac{p-q}{p}}}\|a\|_{L^1(\Omega)} .
$$
Now, we state the second main result of this paper.
\begin{thm}\label{mainresult2}
Let $h \in \mathfrak{F}$ and assume that
\begin{enumerate}
\item[(Bh1)] There exist the positive constants $b$ and $t<p$ such that
$$
H(x, \xi) \leq b\left(1+|\xi|^{t}\right)
$$
for almost every $ x \in \Omega$ and for every $\xi \in \mathbb{R}$; 
\item[(Bh2)] There exist two positive constants $\varepsilon$ and $\delta$, with $\delta>\varepsilon \kappa$ such that
$$
\frac{\int_{\Omega} H(x, \delta) d x}{\delta^p}>a_1 \frac{L_1}{\varepsilon^{p-1}}+a_2 L_2 \varepsilon^{q-p}
$$
\end{enumerate} 
Then, for each parameter $\lambda$ belonging to
$$
\left( \frac{\delta^p\|a\|_{L^1(\Omega)}}{p \int_{\Omega} H(x, \delta) d x}, \frac{\|a\|_{L^1(\Omega)}}{p\left(a_1 \frac{L_1}{\varepsilon^{p-1}}+a_2 L_2 \varepsilon^{q-p}\right)} \right)
$$
the problem \eqref{P1} possesses at least three weak solutions in $W^{s,p}_{\Omega}.$
\end{thm}

\begin{rem}
	 It is important to emphasize that obtaining explicit upper bounds for the constants $c$ or $c_q$ will provide a sufficient condition to ensure the hypotheses (Ah2) or (Bh2) in Theorem \ref{mainresult1} and Theorem \ref{mainresult2}, respectively. While the attainability of Sobolev constants $c$ or  $ c_q $ is well-established in the literature, obtaining their explicit value is generally a challenging problem. In fact, it is a general fact that the explicit solutions to the most of the nonlinear equations are not yet known.
\end{rem}

\begin{rem}
We observe that Theorem \ref{mainresult1} holds true when $N < ps,$ but does not apply when $N \geq ps.$ Specifically, Theorem \ref{mainresult1} cannot be applied to Example \ref{example1}. However, by using Theorem \ref{mainresult2}, one can verify the existence of three solutions to Example \ref{example1} (see Section 3 for the details).
\end{rem}

\noindent The main tool used to prove Theorem \ref{mainresult1} and Theorem \ref{mainresult2}  are the critical points theorems which we recall here.
\begin{thm}\label{threecrit1}(Theorem B \cite{Averna1})
Let $M$ be a reflexive real Banach space and  $M^*$ denotes the dual space of $M.$ Let  $T: M \rightarrow \mathbb{R}$ be a continuously G\^{a}teaux differentiable and sequentially weakly lower semicontinuous functional whose G\^{a}teaux derivative admits a continuous inverse on $M^*,$ also let $ S: M \rightarrow$ $\mathbb{R}$ be a continuously Gâteaux differentiable functional whose G\^{a}teaux derivative is compact. Assume that
\begin{enumerate}
\item[(a1)] $T-\lambda S$ is coercieve for all $\lambda \in (0, \infty).$ \\
\item[(a2)] there exist $r \in \mathbb{R},$ such that $\inf_{M} T <r,$ and $A_{1}(r)<A_{2}(r),$ where $$A_{1}(r):= \inf_{u \in T^{-1}\left((-\infty, r)\right)} \frac{\sup_{\overline{T^{-1}\left((-\infty, r)\right) }^{\omega}} S- S(u) }{r-T(u)}, $$   $$A_{2}(r):= \inf_{u \in T^{-1}\left((-\infty, r)\right)} \sup_{v \in T^{-1}([r, \infty)} \frac{S(v)-S(u)}{T(v)-T(u)},$$ and $\overline{T^{-1}\left((-\infty, r)\right) }^{\omega}$ is the closure of $T^{-1}\left((-\infty, r)\right)$ in the weak topology. Then, for each $\lambda \in \left(\frac{1}{A_1(r)},\, \frac{1}{A_{2}(r)} \right)$ the functional $T-\lambda S$ has at least three critical points in $M.$ 
\end{enumerate}
\end{thm}

\begin{thm}\label{Threecrit2} (\cite{Bonano1}, Theorem 3.2 \cite{Bonano})
	Let $M$ be a reflexive real Banach space and  $M^*$ denotes the dual space of $M.$ Let  $T: M \rightarrow \mathbb{R}$ be a coercive, continuously G\^{a}teaux differentiable and sequentially weakly lower semicontinuous functional whose G\^{a}teaux derivative admits a continuous inverse on $M^*,$ also let $ S: M \rightarrow$ $\mathbb{R}$ be a continuously Gâteaux differentiable functional whose G\^{a}teaux derivative is compact such that
	$$
	T(0)=S(0)=0 .
	$$
	Assume that there exist $\rho >0$ and $x_0 \in M$, with $\rho <T(x_0)$, such that:\\
	\begin{enumerate}
		\item[(b1)] $ \frac{\sup _{T(x) \leq \rho} S(x)}{\rho}<\frac{S(x_0)}{T(x_0)}$ \\
		\item[(b2)] For each $\lambda \in \Lambda_{\rho}:=\left(\frac{T(x_0)}{S(x_0)}, \frac{\rho}{\sup _{T(x) \leq \rho} S(x)} \right) $ the functional $T-\lambda S$ is coercive.
	\end{enumerate}
	Then, for each $\lambda \in \Lambda_{\rho}$, the functional $T-\lambda S$ has at least three distinct critical points in $M.$
\end{thm} 
\noindent We will now prove our main results in the next two sections.
\section{\textbf{The Case $N<sp$}}
  \noindent Let $M:=W^{s,p}_{\Omega}$ and consider the functionals $T, S: M \rightarrow \mathbb{R}$ defined as follows

$$
T(u):=\frac{\|u\|^p}{p}= \frac{1}{p}\left( \int_{\Omega} a(x) \abs{u}^p dx + \frac{1}{2}\int_{T(\Omega)} \frac{|u(x)-u(y)|^{p}}{|x-y|^{N+sp}} dx d y \right),
$$ and 
$$\quad S(u):=\int_{\Omega} H(x, u(x)) d x, \quad \forall u \in M.$$ 
Since $N<sp,$ $M$ is compactly embedded in $C({\overline{\Omega}})$  implies that $T$ and $S$ are well defined and continuously G\^{a}teaux differentiable functionals with
\begin{align*}
	& T^{\prime}(u)(v)= \frac{1}{2} \int_{T(\Omega)} \frac{\abs{u(x)-u(y)}^{p-2}(u(x)-u(y))(v(x)-v(y))}{|x-y|^{N+p s}} d x d y +  \int_{\Omega} a(x)\abs{u}^{p-2}u v dx,\\
	& S^{\prime}(u)(v)=\int_{\Omega} h(x, u(x)) v(x) d x, 
\end{align*}
for every $u, v \in W^{s,p}_{\Omega}, $ also $S^{\prime}$ is compact. Fix $\lambda>0$. A critical point of the functional $J_{\lambda}=T- \lambda S$ is a function $u \in W^{s,p}_{\Omega}$ such that

$$
T^{\prime}(u)(v)-\lambda S^{\prime}(u)(v)=0
$$
for every $v \in W^{s,p}_{\Omega}.$ Hence, the critical points of the functional $J_{\lambda}$ are weak solutions of  \eqref{P1}. 

\begin{lem}
$T^{\prime}$ admits a continuous inverse on $M^{*}.$
\end{lem}
\begin{proof}
Let $x,y \in \mathbb{R}^N,$ and $\langle  \,, \rangle $ denotes the usual inner product on $\mathbb{R}^N.$ Then, by using Equation 2.2 from \cite{Simon}, there exists a positive constant \( D(p) \) such that
\begin{align}
\langle \abs{x}^{p-2}x-\abs{y}^{p-2}y, x-y \rangle \geq D(p) \abs{x-y}^p.
\end{align}
Thus, one can easily verify that $$\left(T^{\prime}(u)-T^{\prime}(v) \right)(u-v) \geq D(p) \norm{u-v}^{p} \quad \forall u,v \in M.$$  Therefore, $T^{\prime}$ is a uniformly monotone operator on $M.$ Also, since $\ess \inf_{\Omega} a(x)>0,$ standard arguments enure that $T^{\prime}$ is coercieve and semicontinuos on $M.$ Therefore, by applying Theorem 26.A. from \cite{zie}, it follows that $ T^{\prime} $ has a continuous inverse on $ M^* .$
\end{proof}
\noindent \textbf{Proof of Theorem \ref{mainresult1}:} We use Theorem \ref{threecrit1} to prove the Theorem \ref{mainresult1}. From the above Lemma, we know that $T^{\prime}$ has a continuous inverse on $M^{*}.$
It is well-known that the Galiardo-seminorm $\int_{T(\Omega)} \frac{|u(x)-u(y)|^{p}}{|x-y|^{N+sp}} dx d y $ and the $L^{p}$-norms are sequentially weakly lower semicontinuous.  Therefore, $T$ is sequentially weakly lower semicontinuous. Note that, due to the assumption (Ah1) of Theorem \ref{mainresult1}, hypothesis (a1) of Theorem \ref{threecrit1} is satisfied, meaning that $ T - \lambda S $ is coercive for all $ \lambda \in (0, \infty).$ Now, in order to prove the hypothesis (a2) of Theorem \ref{threecrit1}, we let 
$r:=\frac{1}{p}\left(\frac{\gamma}{c}\right)^{p}.$ From the assumption (Ah2), we get
\begin{align*}
\frac{\Gamma}{c^p \norm{a}_{L^{1}(\Omega)}\gamma^p}= \frac{\int_{\Omega} \sup_{\abs{\xi}\leq\gamma } H(x,\xi) dx}{c^p \norm{a}_{L^{1}(\Omega)}\gamma^p} & < \frac{1}{c^p \norm{a}_{L^{1}(\Omega)}\left(1+c^p \norm{a}_{L^{1}(\Omega)} \right) } \frac{\int_{\Omega} H(x,\eta) dx}{\eta^p} \\
&= \left(\frac{1}{c^p \norm{a}_{L^{1}(\Omega)}}-\frac{1}{1+c^p \norm{a}_{L^{1}(\Omega)} }  \right) \frac{\int_{\Omega} H(x,\eta) dx}{\eta^p},
\end{align*}
then
\begin{align*}
\frac{1}{1+c^p \norm{a}_{L^{1}(\Omega)} } \frac{\int_{\Omega} H(x,\eta) dx}{\eta^p} + \frac{\Gamma}{c^p \norm{a}_{L^{1}(\Omega)}} < \frac{1}{c^p \norm{a}_{L^{1}(\Omega)}} \frac{\int_{\Omega} H(x,\eta) dx}{\eta^p} 
\end{align*}
thus, being $\gamma <\eta,$ we have
\begin{align*}
	\frac{1}{1+c^p \norm{a}_{L^{1}(\Omega)} } \frac{\int_{\Omega} H(x,\eta) dx}{\eta^p}  < \frac{\int_{\Omega} H(x,\eta) dx-\Gamma}{c^p \norm{a}_{L^{1}(\Omega) }\eta^p}
\end{align*}
therefore, by using (Ah2) again, we obtain
\begin{align*}
\frac{\Gamma}{\gamma^p} < \frac{\int_{\Omega} H(x,\eta) dx-\Gamma}{c^p \norm{a}_{L^{1}(\Omega) }\eta^p} 
\end{align*}
from which, multiplying by $p c^p,$ we obtain 
\begin{align}\label{pcp}
	\frac{p c^p  \Gamma}{\gamma^p} < p \frac{\int_{\Omega} H(x,\eta) dx-\Gamma}{ \norm{a}_{L^{1}(\Omega) }\eta^p} .
\end{align}
We claim that:
\begin{align}\label{claim1}\tag{Claim I}
A_{1}(r) \leq \frac{p c^p  \Gamma}{\gamma^p}
\end{align}
and 
\begin{align}\label{claim2} \tag{Claim II}
A_2(r) \geq p \frac{\int_{\Omega} H(x,\eta) dx-\Gamma}{ \norm{a}_{L^{1}(\Omega) }\eta^p},
\end{align}
from which hypothesis (a2) of Theorem \ref{threecrit1} follows. \\
To prove the \eqref{claim1}, first note $u=0$ belongs to $T^{-1}(-\infty, r),$ and that $S(0)=0,$ we get 

\begin{align*}
	A_1(r) \leq \frac{\sup_{\overline{T^{-1}(-\infty, r)}^{\omega}} \int_{\Omega} H(x,u(x))dx }{r},
\end{align*}
and, since $\overline{T^{-1}(-\infty, r)}^{\omega}= {T^{-1}(-\infty, r]},$ we have 
\begin{align*}
\frac{\sup_{\overline{T^{-1}(-\infty, r)}^{\omega}} \int_{\Omega} H(x,u(x))dx }{r} = \frac{\sup_{{T^{-1}(-\infty, r]}} \int_{\Omega} H(x,u(x))dx }{r}.
\end{align*}
Therefore, noting that  $\abs{u(x)} \leq c (pr)^{\frac{1}{p}}=\gamma,$ for every $u \in M$ such that $T(u) \leq r$ and for each $x \in \Omega,$ we obtain 
\begin{align*}
\frac{\sup_{{T^{-1}(-\infty, r]}} \int_{\Omega} H(x,u(x))dx }{r} \leq \frac{\int_{\Omega} \sup_{\abs{\xi} \leq \gamma }H(x,\xi) dx}{r}= \frac{\Gamma}{r}.
\end{align*}
So, from the definition of $r,$ the (\ref{claim1}) follows. Next, we prove the (\ref{claim2}). For each $v \in M$ with $T(v) \geq r,$ we have 
\begin{align*}
	A_2(r) &\geq \inf_{v \in T^{-1}(-\infty, r)} \frac{\int_{\Omega}H(x,v(x)) dx-\int_{\Omega}H(x,u(x)) dx }{T(v)-T(u)} \\
	& \geq \inf_{v \in T^{-1}(-\infty, r)} \frac{\int_{\Omega}  H(x,v(x)) dx-\int_{\Omega}\sup_{\abs{\xi}\leq \gamma} H(x,u(x)) dx }{T(v)-T(u)} ,
\end{align*}
from which, being $0<T(v)-T(u) \leq T(v)$ for every $u \in T^{-1}((-\infty, r)),$ and under further condition 
\begin{align}\label{HGamma}
\int_{\Omega}H(x,v(x)) dx \geq \int_{\Omega} \sup_{\abs{\xi}\leq \gamma} H(x,u(x)) dx=\Gamma,
\end{align}
we can write
\begin{align*}
\inf_{v \in T^{-1}(-\infty, r)} \frac{\int_{\Omega}  H(x,v(x)) dx-\int_{\Omega}\sup_{\abs{\xi}\leq \gamma} H(x,u(x)) dx }{T(v)-T(u)} \geq p  \frac{\int_{\Omega}  H(x,v(x)) dx-\int_{\Omega}\sup_{\abs{\xi}\leq \gamma} H(x,u(x)) dx }{\norm{v}^{p}}.
\end{align*} 
If we put $v(x):=\eta,$ for each $x \in \Omega,$ we have $\norm{v}=\norm{a}_{L^{1}(\Omega)}^{\frac{1}{p}}\eta,$ hence, by (\ref{bestconst}) and $\gamma < \eta,$ we get $T(v)>r.$ Moreover, with this choice of $v,$ (\ref{pcp}) ensures (\ref{HGamma}), thus (\ref{claim2}) is also proved. Hence, the conclusion follows from Theorem \ref{mainresult1}, noting that
\begin{align*}
\frac{1}{A_{2}(r)} \leq \frac{\norm{a}_{L^{1}(\Omega)}\eta^p}{p\left(\int_{\Omega} H(x,\eta) dx- \Gamma \right) }
\end{align*}
and
 \begin{align*}
 \frac{1}{A_{2}(r)}  \geq \frac{\gamma^p}{c^p p \Gamma}.
 \end{align*}
This completes the proof of Theorem \ref{mainresult1} \qed
\section{\textbf{The Case $N \geq sp$}}
Let $M, T$ and $S$ are the same as in Section 2, i.e., $M=W^{s,p}_{\Omega},$
$$
T(u):=\frac{\|u\|^p}{p}= \frac{1}{p}\left( \int_{\Omega} a(x) \abs{u}^p dx + \int_{T(\Omega)} \frac{|u(x)-u(y)|^{p}}{|x-y|^{N+sp}} dx d y \right),
$$ and 
$$\quad S(u):=\int_{\Omega} H(x, u(x)) d x, \quad \forall u \in M.$$ We now prove the Theorem \ref{mainresult2}.\\
\noindent\textbf{Proof of Theorem \ref{mainresult2}:}
The continuity and compactness of the embedding $ W^{s,p}_{\Omega} \hookrightarrow L^q(\Omega) $ ensure that $S$ is well-defined, continuously G\^{a}teaux differentiable, and has a compact derivative. On the other hand, it is easy to see that $ T(u) \to + \infty $ as $ \|u\| \to +\infty ,$ which implies that $ T: M \to \mathbb{R} $ is a coercive functional. Further, as shown in Section 2, it is easy to verify that $ T $ is continuously G\^{a}teaux differentiable, weakly sequentially lower semi-continuous, and its Gâteaux derivative has a continuous inverse on \( M^* \).  Moreover, $T(0)=S(0)=0.$ Since $h \in \mathfrak{F}$, one has that
\begin{align}\label{Carot1}
H(x, \xi) \leq a_1|\xi|+a_2 \frac{\abs{\xi}^q}{q}
\end{align}
for every $(x, \xi) \in \Omega \times \mathbb{R}.$
Let $\rho \in \left( 0, +\infty \right)$ and consider the function

$$
g(\rho)=\frac{\sup _{u \in T^{-1}(-\infty, \rho]}S(u)}{\rho}
$$
Taking into account \eqref{Carot1}, it follows that
$$
S(u)=\int_{\Omega} H(x, u(x)) d x \leq a_1\|u\|_{L^{1}(\Omega)}+\frac{a_2}{q}\|u\|^{q}_{L^{q}(\Omega}.
$$
Then, for every $u \in W^{s,p}_{\Omega}$ such that $T(u) \leq \rho$, owing to \ref{sobolevconst}, we get,

\begin{align}
	S(u) \leq \left( \left(p \rho \right)^{\frac{1}{p}} \right) c_1a_1 + \frac{p^{\frac{q}{p}}c_{q}^{q}a_2}{q}\rho^{\frac{q}{p}}.
\end{align}
Hence, 
\begin{align}\label{Psibound}
\sup _{u \in T^{-1}(-\infty, \rho]}\Psi(u) \leq \left( \left(p \rho \right)^{\frac{1}{p}} \right) c_1a_1 + \frac{p^{\frac{q}{p}}c_{q}^{q}a_2}{q}\rho^{\frac{q}{p}}.
\end{align}
From \eqref{Psibound}, the following inequality holds
\begin{align}\label{gbound}
g(\rho) \leq \left(\frac{p}{\rho^{p-1}}\right)^{1 / p} c_1 a_1+ \frac{p^{\frac{q}{p}}c_{q}^{q}a_2}{q}\rho^{\frac{q}{p}-1}
\end{align}
for every $\rho>0.$ Next, put $u_\delta(x)=\delta$ for every $x \in \mathbb{R}^N.$ Clearly $u_{\delta} \in W^{s,p}_{\Omega}$ and we have
\begin{align}\label{Phid}
	T\left(u_{\delta} \right) & =  \frac{1}{p} \left(  \frac{1}{2}\int_{T(\Omega)} \frac{|u_{\delta}(x)-u_{\delta}(y)|^{p}}{|x-y|^{N+sp}} dx d y +  \int_{\Omega} a(x)\abs{u_{\delta}}^p dx \right) \nonumber \\
	& =\frac{1}{p} \int_{\Omega} a(x) \delta^P d x=\frac{\delta^p}{p}\|a\|_{L^{1}(\Omega)}
\end{align}
Taking into account that $\delta>\varepsilon \kappa$, by a direct computation, one has $\varepsilon^p<T \left(u_\delta\right)$. Moreover,
\begin{align}\label{Psid}
	S(u_\delta)=\int_{\Omega} H\left(x, u_\delta(x)\right) d x=\int_{\Omega} H(x, \delta) d x
\end{align}
Hence, from \eqref{Phid} and \eqref{Psid}, one has 
\begin{align}\label{ratiophipsi}
\frac{S \left(u_\delta \right)}{T\left(u_\delta\right)}=\frac{\int_{\Omega} H(x, \delta) d x}{\delta^{p}\|\alpha\|_{L^{1}(\Omega)}} p.
\end{align}
In view of (Bh2) and taking into account \eqref{gbound} and \eqref{ratiophipsi}, we get
\begin{align*}
	g(\varepsilon^p)=\frac{\sup _{u \in T^{-1}(-\infty, \varepsilon^p]}S(u)}{\varepsilon^p} & \leq \frac{p^{1/p}}{\varepsilon^{p-1}} c_1 a_1+ \frac{p^{\frac{q}{p}}c_{q}^{q}a_2}{q}\varepsilon^{q-p}\\
	& =\frac{p}{\|a\|_{L^1(\Omega)}}\left(a_1 \frac{K_1}{\varepsilon^{p-1}}+a_2 K_2 \varepsilon^{q-p}\right) \\
	& <\frac{\int_{\Omega} F(x, \delta) d x}{\delta^p\|\alpha \|_{L^{1}(\Omega)}} p \\
	& =\frac{S\left(u_\delta \right)}{T\left(u_\delta \right)}. 
\end{align*}
Therefore, assumption $\left(b_1\right)$ of Theorem \ref{Threecrit2} is satisfied.
Moreover, since $t<p$, from the H\"{o}lder's inequality, we obtain
$$
\int_{\Omega}|u(x)|^t d x \leq\|u\|_{L^{p}(\Omega)}^{t} \abs{\Omega}^{\frac{p-t}{p}}, \quad  \forall u \in W^{s,p}_{\Omega}  .
$$
Therefore, bearing in mind \eqref{sobolevconst}, we obtain
\begin{align}\label{utbound}
\int_{\Omega}|u(x)|^t d x \leq c_{p}^{t}  \|u\|^t \abs{\Omega}^{\frac{p-t}{p}}, \quad  \forall u \in W^{s,p}_{\Omega} .
\end{align}
So, from \eqref{utbound} and condition (Bh1), it follows that
$$
J_\lambda(u) \geq \frac{\|u \|}{p}-\lambda b c_{p}^{t}  \|u\|^t \abs{\Omega}^{\frac{p-t}{p}} -\lambda b \abs{\Omega}, \quad \forall u \in W^{s,p}_{\Omega}  .
$$
Hence, $J_{\lambda}$ is a coercive functional for every positive parameter, in particular, for $\lambda \in \Lambda_{(\varepsilon,\delta)}= \left( \frac{T\left(u_\delta \right)}{S\left(u_\delta\right)},\, \frac{\varepsilon^p}{\sup_{T(u) \leq \varepsilon^p} S(u)}  \right),$ for which also condition (b2) holds. Hence, all the assumptions of Theorem \ref{Threecrit2} are satisfied. Then, for each $\lambda \in$ $\Lambda_{(\varepsilon, \delta)}$ the functional $J_\lambda$ has at least three distinct critical points that are weak solutions of the problem \eqref{P1}. This completes the proof of Theorem \ref{mainresult2}. \qed \\

\noindent The following corollary is a direct consequence of Theorem \ref{mainresult2}.
\begin{cor}\label{result2cor}
Let $\phi \in L^{\infty}(\Omega)$ be a non-negtive and non-zero function  and let $\psi: \mathbb{R} \rightarrow \mathbb{R}$ be a non-negative continuous function such that
$\psi(0) \neq 0$. Put $\Psi(\xi):=\int_0^{\xi} \psi(t) d t$ for every $\xi \in \mathbb{R}$ and assume that\\
\begin{enumerate}
\item[($\phi 1$)]  There exist two non-negative constants $a_1, a_2$ such that for every $t \in \mathbb{R}$ one has
$$
\psi(t) \leq a_1+a_2|t|^{q-1}
$$
for some $q \in \left( 1, \frac{p N }{N-ps} \right)$ if $1 \leq sp<N$ and $1<q<+\infty$ if $sp=N$;
\item[$(\phi 2)$] There exists $\delta>\kappa$ such that
$$
\frac{\Psi(\delta)}{\delta^ p}>\frac{\|\phi\|_{L^{\infty}(\Omega)}}{\|\phi\|_{L^{1}(\Omega)}}\left(a_1 L_1+a_2 L_2\right) ;
$$
\item[$(\phi 3)$] $$
\lim _{t \rightarrow+\infty} \frac{h(t)}{t^\beta}=0
$$
for some $0 \leq \beta<(p-1)$.
\end{enumerate}
Then, for each parameter $\lambda$ belonging to
$$
\left( \frac{\delta^p\| \alpha \|_{L^1(\Omega)}}{p\|\phi\|_{L^1(\Omega)} \Psi(\delta)}, \frac{\| \alpha \|_{L^1(\Omega)}}{p\|\phi\|_{L^{\infty}(\Omega)}\left(a_1 L_1+a_2 L_2\right)}  \right) 
$$
the problem
\begin{align}\tag{P2}
\begin{cases} (-\Delta)^{s}_p u+\alpha(x)|u|^{p-2} u=\lambda \phi(x) \psi(u) & \text { in } \Omega \\ \mathcal{N}_{s,p}u=0 & \text { on } \mathbb{R}^{N}\setminus \Omega \end{cases}
\end{align}
possesses at least three non-zeno weak solutions in $W^{s,p}_{\Omega}.$
\end{cor}
\noindent We conclude this section with an application of Corollary \ref{result2cor}.
\begin{ex}\label{example1}
	Let $\Omega \subset \mathbb{R}^N$ be a bounded, open set of class $C^{1}.$ Let $N>sp>1$ and $a \in L^{\infty}(\Omega),$ with $\ess\inf_{x \in \Omega} a(x) >0.$ Let $q \in \left(p, \frac{Np}{N-ps}\right)$ if $N>ps>1$ or $q>p$ if $N=ps.$ Let $h: \mathbb{R} \rightarrow \mathbb{R}$ defined  by
	 \begin{align}
	\psi(t):= \begin{cases}
		1+\abs{t}^{q-1} &  \text{ if } t \leq \rho,\\
		\frac{(1+\rho^2)(1+\rho^{q-1})}{1+t^2} & \text{ if } t>\rho,
	\end{cases}
	\end{align}
	where $\rho$ is a fixed constant such that 
	\begin{align}\label{rhobound}
	\rho > \max \left\{\kappa, q^{\frac{1}{q-p}} \left( \frac{L_1+L_2}{\abs{\Omega}} \right)^{\frac{1}{q-p}} \right\},
	\end{align}
	and let $\phi(x)=1$ for all $x \in \Omega.$ Clearly, $\psi(0) \neq 0$ and $\psi(t) \leq \left(1+\abs{t}^{q-1}\right)$ for every $t \in \mathbb{R}.$  Hence the condition ($\phi1$) is satisified for for $a_1=a_2=1.$ Moreover, also condition $(\phi 3)$ is verified. Indeed, for every fixed $\beta \in (0,p-1),$ one has that $\lim _{t \rightarrow+\infty} \frac{h(t)}{t^\beta}=0 .$ Finally, using \eqref{rhobound}, it follows that 
	$$\frac{L_1+L_2}{\abs{\Omega}}< \frac{\rho^{q-p}}{q} .$$ Therefore, 
	$$\frac{\int_{0}^{\rho}\psi(t) dt }{\rho^p} =\frac{\rho^{q-p}}{q}+\frac{1}{\rho^{p-1}}>\frac{L_1+L_2}{\abs{\Omega}}$$ and condition $(\phi 2)$ holds choosing $\delta=\rho.$ Consequently, owing to Corollary \ref{result2cor}, the following problem
		\begin{align}  \begin{cases}
			(-\Delta)^{s}_{p}u + a(x) \abs{u}^{p-2}u =\lambda \psi(u) & \text {in } \Omega, \\
			\mathcal{N}_{s,p}u=0  & \text {in } \mathbb{R}^N \setminus \overline{\Omega}, 
		\end{cases}
	\end{align}
	for each parameter $\lambda$ belonging to $$
	\left( \frac{\rho^p\| a \|_{L^1(\Omega)}}{p \abs{\Omega} \psi(\rho)}, \frac{\| \ a \|_{L^1(\Omega)}}{p \left( L_1+ L_2\right)}  \right) ,
	$$ possesses at least three non-zero weak solutions in $W^{s,p}_{\Omega}.$
\end{ex}

\end{document}